\documentclass[a4paper, 12pt]{article}

\usepackage{amsfonts, amsmath, amsthm}
\usepackage{mathrsfs, latexsym}
\usepackage{bbm}
\usepackage{mathtools}

\allowdisplaybreaks
\theoremstyle{plain}
\newtheorem{thm}{Theorem}[section]

\newtheorem{lem}{Lemma}[section]

\theoremstyle{definition}

\theoremstyle{remark}
\newtheorem{rmk}{Remark}[section]

\theoremstyle{remark}
\newtheorem{exm}{Example}[section]

\numberwithin{equation}{section}

\newcommand{\R}{\mathbb{R}}
\newcommand{\C}{\mathbb{C}}

\renewcommand{\L}{\mathcal L}
\renewcommand{\P}{\mathcal P}
\renewcommand{\hat}{\widehat}
\newcommand{\F}{\mathcal F}
\newcommand{\IF}{\mathcal F^{-1}}
\newcommand{\J}{\mathcal J}
\newcommand{\M}{\mathcal M}

\newcommand{\D}{\mathcal D}
\newcommand{\ID}{\mathcal D^{-1}}
\newcommand{\B}{\mathcal B}
\newcommand{\IB}{\mathcal B^{-1}}
\newcommand{\V}{\mathcal V}
\newcommand{\IV}{\mathcal V^{-1}}
\newcommand{\W}{\mathcal W}
\newcommand{\IW}{\mathcal W^{-1}}
\newcommand{\G}{\mathcal G}
\newcommand{\Q}{\mathcal Q}
\newcommand{\IQ}{\mathcal Q^{-1}}

\newcommand{\ipax}{\jb{i \pa_x}}
\newcommand{\expm}{e^{-it\jb{i\pa_x}}}
\newcommand{\expp}{e^{it\jb{i\pa_x}}}
\newcommand{\expmxi}{e^{-it\jb{\xi}}}
\newcommand{\exppxi}{e^{it\jb{\xi}}}
\newcommand{\pa}{\partial}
\newcommand{\eps}{\varepsilon}
\newcommand{\jb}[1]{\left\langle #1 \right\rangle}
\newcommand{\conj}[1]{\overline{#1}}
\newcommand{\dal}{\Box}

\DeclareMathOperator{\re}{\rm Re}
\DeclareMathOperator{\im}{\rm Im}

\begin{document}
\title{
A note on a system of cubic nonlinear\\
Klein-Gordon equations in one space dimension
 }

\author{
          Donghyun Kim\thanks{
              Seoul National Science Museum. 
              215, Changgyeonggung-ro, Jongno-gu, Seoul 110-360, Korea
              (E-mail: {\tt u553252d@alumni.osaka-u.ac.jp},
              Tel: {\tt (+82)1026771576})
            }
}

\date{July 26, 2014}   
\maketitle

{\small \noindent{\bf Abstract:}\ 
We study the Cauchy problem for a system of 
cubic nonlinear Klein-Gordon equations in one space dimension. 
Under a suitable structural condition on the nonlinearity, 
we will show that the solution exists globally and decays of the
order $O(t^{-1/2})$ in $L^\infty$ 
as $t$ tends to infinity 
without the condition of a compact support on the Cauchy data
which was assumed in the previous works.\\}





\section{Introduction} \label{sec_intro}
In the present paper, we consider the Cauchy problem for 
a system of nonlinear Klein-Gordon equations
in one space dimension:
\begin{align}\label{nlkg}
 (\dal +1) u = F (u),\qquad (t,x)\in\R\times\R
\end{align}
with initial conditions
\begin{align}\label{data}
 u(0,x) = f(x), \quad \pa_t u(0,x) = g(x), \qquad x\in\R,
\end{align}
where $\dal = \pa_t^2 - \pa_x^2$,
$u = (u_j)_{1\le j\le N}$ is an $\R^N$-valued unknown function of
$(t,x)\in\R\times\R$
and
$f = (f_j)_{1\le j\le N}$, $g = (g_j)_{1\le j\le N}$
are $\R^N$-valued.
The nonlinear term $F(u)=(F_j(u))_{1\le j\le N}$ is assumed to be
a cubic homogeneous polynomial of $u$, i.e., we may put
\begin{align}\label{Fjform}
 F_j(u) = \sum_{k,l,m=1}^N C_{j,k,l,m} u_k u_l u_m
\end{align}
with some real constants $C_{j,k,l,m}$.
Also we assume that the data $f$ and $g$ 
are sufficiently small throughout this paper.
The aim of this paper is to introduce a structural condition on the nonlinearity
under which the solution of \eqref{nlkg}--\eqref{data} exists globally and
decays like $O(t^{-1/2})$ in $L^\infty$ as $t \to +\infty$
without the restriction of a compact support on the Cauchy data
which was assumed in the previous works \cite{Su2005-2}, \cite{Kim1}.

From the perturbative viewpoint, 
cubic nonlinear Klein-Gordon systems are of special interest
in one space dimensions because large-time
behavior of the solution is actually affected by the structure of the nonlinearities.
So we have to put some restrictions on the nonlinearities to obtain 
global existence and sharp time-decays for the solutions.

Let us recall some previous results briefly. 
In order to introduce numerous results, 
we let the cubic nonlinear term $F$ in \eqref{nlkg}
also include first-order derivatives as well as $u$,
i.e., we replace $F(u)$ by $F(u, \pa_t u, \pa_x u)$
throughout this section.
First we restrict our attention to the problem \eqref{nlkg}
with sufficiently smooth, compactly-supported data;
$f, g\in C_0^\infty(\R)$.
Then we can find lots of works to the Cauchy problem
\eqref{nlkg}--\eqref{data} which discuss global existence 
and large-time behavior of the solution.
In the scalar case ($N=1$), some classes of the cubic terms were found
by Moriyama \cite{Mori1997} and Katayama \cite{Kata1999} 
for which the Cauchy problem admits the unique global solution with a free profile.
Delort \cite{De2006} extended their results and found a sufficient condition
that the small data global existence holds for \eqref{nlkg}--\eqref{data}
when $N=1$.
A pointwise asymptotic profile of the solution is also provided in \cite{De2006}.
In the most common case; $F=u^3$,
we can conclude from \cite{De2006} that the global small solution to
$(\dal+1)u=u^3$
decays like a free solution, i.e. of the order $O(t^{-1/2})$ in $L^\infty$,
while it does not behave like a free solution.
Sunagawa \cite{Su2005-2} partially extended this result
to the corresponding two-component system 
\begin{align}\label{nlkg-res}
 \left\{
 \begin{array}{l}
 (\dal+1)u_1 = (u_1^2 + u_2^2) u_1,\\
 (\dal+1)u_2 = (u_1^2 + u_2^2) u_2.
 \end{array}
 \right.
\end{align}
It should be noted that viewing $U=u_1 + iu_2$, $i=\sqrt{-1}$,
this system can be rewritten as the complex-valued single equation 
\begin{align}\label{nlkg-comp}
(\dal+1)U = |U|^2 U.
\end{align}
According to \cite{Su2005-2},
the small data global existence holds for \eqref{nlkg-comp}
and the solution decays like $O(t^{-1/2})$ in $L^\infty$,
however, only a kind of a prediction is given there
for the asymptotic profile of the solution
since the approach \cite{De2006} does not work well in the system case
as pointed out in \cite{Su2005-2}.
As far as the author knows, 
it is still the open problem to find the pointwise asymptotics
of the solution for \eqref{nlkg-comp},
even in the case that the Cauchy data are sufficiently small, smooth
and compactly-supported.
In \cite{Su2006}, the nonlinear dissipative term $F=-(\pa_t u)^3$
was investigated in detail and the asymptotic profile of the solution was given
from which the sharp decay rate $O(t^{-1/2} (\log t)^{-1/2})$ in $L^\infty$
follows immediately.
Recently, this result has been extended by \cite{Kim1} to the
complex-valued case, i.e., 
$(\dal +1)U = -|\pa_t U|^2 \pa_t U$.
Moreover, a structural condition on the nonlinearity is given in \cite{Kim1}
under which the solution for \eqref{nlkg} exists globally 
and decays rapidly of the order $O(t^{-1/2} (\log t)^{-1/2})$ in $L^\infty$.
However, the situation is the same as in the case of \eqref{nlkg-comp},
that is, the previous approach \cite{Su2006} does not work well
in the system case if we try to find an asymptotic profile of the solution.
We refer the readers to 
\cite{Fang2006, Kim2, Lind2005-2, Su2005}
and the references cited therein for more related results 
concerning global existence and large-time behavior
of small amplitude solutions to the Cauchy problem
\eqref{nlkg} for sufficiently smooth, compactly-supported data.

Now we turn our attention to the problem \eqref{nlkg}--\eqref{data}
with the condition that the data are not compactly-supported.
In this case, there are not so many papers
which treat the problem
with cubic nonlinearities in one space dimension,
because the methods used in the papers mentioned above
(e.g. change of variables using the hyperbolic coordinates)
are not applicable in general.
In the scalar case, 
Hayashi \cite{Ha2008} proved the global existence of the small solution
for \eqref{nlkg}--\eqref{data} and found
some asymptotic profiles of the solution when $F = u^3$,
if the data are in suitable weighted Sobolev spaces.
However, the pointwise asymptotic behavior of the solution 
is still not found in \cite{Ha2008}.
The main tool of \cite{Ha2008} is a decomposition of the
free Klein-Gordon evolution group, 
which was used also in
\cite{Ha2009} to treat a final state problem for \eqref{nlkg} when $N=1$
(see also \cite{Ha2010, Ha2012, Ha2013} 
for further developments of this method).
In the system case, 
when the linear part of \eqref{nlkg} is replaced by $\dal +m_j^2$,
a structural condition of the nonlinearities $F_j$ and
the masses $m_j$ for small data global existence 
with a free profile was studied by Sunagawa \cite{Su2003},
under the conditions that the data are in suitable Sobolev spaces.
However the results in \cite{Su2003} does not cover the resonant cases,
for example \eqref{nlkg-res}.

It should be also remarked that when $N\ge 2$,
normally we cannot expect even the decay rate of a free solution
for a certain global solution;
for instance, the following system
\begin{align*}
 \left\{
 \begin{array}{l}
 (\dal+1)u_1 = 0,\\
 (\dal+1)u_2 = u_1^3
 \end{array}
 \right.
\end{align*}
admits a global solution, but according to \cite{Su2005}
(see also \cite{Su2004} for more examples including 2-dimensional cases as well), 
the second component $u_2$ decays no faster than
$O(t^{-1/2} \log t)$ in $L^\infty$,
even if the data are sufficiently small, smooth and compactly-supported.

The present paper extends the result in \cite{Ha2008} 
to the complex-valued case \eqref{nlkg-comp}.
However, as in the case of compactly-supported data,
the method in \cite{Ha2008} does not work well
if we try to investigate the asymptotic behavior of the solution when $N\ge 2$.
This paper also can be seen as an extension of the papers
\cite{Kim1}, \cite{Su2005-2}
since the compactness of the Cauchy data has been removed. 
We give a structural condition on the nonlinearities
under which the small data solution exists globally for \eqref{nlkg}--\eqref{data}
and decays like a free solution,
if the data are in suitable weighted Sobolev spaces.
As an application,
we will see that the global complex-valued solution to
\eqref{nlkg-comp} decays like $O(t^{-1/2})$ in $L^\infty$,
if the data are small enough and belong to suitable weighted Sobolev spaces.

\section{Main Results} \label{sec_main}
In order to state the main results, we introduce some notations here.
We denote the usual Lebesgue space by $L^p(\R)$
with the norm
$\left\|\phi\right\|_{L^p} = \left(\int_\R|\phi(x)|^p\,dx\right)^{1/p}$
if $p\in [1,\infty)$ and $\left\|\phi\right\|_{L^\infty} = \sup_{x\in\R}|\phi(x)|$
if $p=\infty$.
The weighted Sobolev space is defined by
\begin{align*}
 H_p^{s,q} (\R) = \left\{\phi = (\phi_1, \ldots, \phi_N)
 \in L^p(\R) : 
 \left\|\phi\right\|_{H_p^{s,q}}
 =
 \sum_{j=1}^N \left\|\phi_j\right\|_{H_p^{s,q}}<\infty\right\}
\end{align*}
with the norm
\begin{align*}
 \left\|\phi_j\right\|_{H_p^{s,q}} 
 = 
 \left\|\jb{x}^q \jb{i\pa_x}^s\phi_j\right\|_{L^p}
\end{align*}
for $s,q \in \R$ and $p\in[1,\infty]$,
where $\jb{\cdot} = \sqrt{1+\cdot\,^2}$.
For simplicity, we write $H^{s,q} = H_2^{s,q}$,
$H_p^s = H_p^{s,0}$
and the usual Sobolev space $H^s = H_2^s$.
We define the Fourier transform $\hat\phi(\xi)$ of a function $\phi(x)$ by
\begin{align*}
 (\F \phi)(\xi) = \hat \phi(\xi) 
 = 
 \frac{1}{\sqrt{2\pi}} \int_\R e^{-ix\xi}\phi(x)\,dx.
\end{align*}
Then the inverse Fourier transform is given by
\begin{align*}
 (\IF\phi)(x)
 = 
 \frac{1}{\sqrt{2\pi}} \int_\R e^{ix\xi}\phi(\xi)\,d\xi.
\end{align*}
We introduce a new function 
$\widetilde F=(\widetilde F_j)_{1\le j\le N} : \C^N \to \C^N$ defined by
\begin{align}\label{tildeF}
 \widetilde{F}_j(Y) 
 = 
 \sum_{k,l,m=1}^N C_{j,k,l,m} \Big(
 \conj{Y_k}Y_l Y_m
 +Y_k \conj{Y_l} Y_m
 +Y_k Y_l \conj{Y_m}\Big)
\end{align}
for $Y=(Y_n)_{1\le n\le N}\in\C^N$ with the constants $C_{j,k,l,m}$
given in \eqref{Fjform}.
Also we denote by $Y\cdot Z$ the standard scalar product in $\C^N$
for $Y, Z\in \C^N$ and write $|Y| = \sqrt{Y\cdot Y}$ as usual.

Now we state the main results.

\begin{thm}\label{thm1}
Let $f\in H^{4,1}$, $g\in H^{3,1}$ and
$\left\|f\right\|_{H^{4,1}} + \left\|g\right\|_{H^{3,1}} = \eps$.
Assume there exists an $N\times N$ positive Hermitian matrix $A$ such that
\begin{align}\label{struc-condi}
 \im \big( AY \cdot \widetilde{F}(Y) \big)
 =
 0
\end{align}
for all $Y\in\C^N$.
Then there exists $\eps_0>0$ such that for all $\eps \in (0, \eps_0)$,
the Cauchy problem \eqref{nlkg}--\eqref{data} admits a unique global solution
\begin{align*}
 u(t)\in C^0([0,\infty); H^{4,1}) \cap C^1 ([0,\infty); H^{3,1})
\end{align*}
satisfying the time-decay estimate
\begin{align}\label{thm1-decay}
 \left\| u(t)\right\|_{H_\infty^1}
 \le
 C\eps (1+t)^{-1/2}
\end{align}
for all $t\ge 0$
with some positive constant $C$ which does not depend on $\eps$.
\end{thm}

Here we introduce a typical example of the nonlinearity
satisfying the condition \eqref{struc-condi}.

\begin{exm}
We consider the following nonlinear Klein-Gordon system
\begin{align}\label{sysexm1}
 \left\{
 \begin{array}{l}
  (\Box+1)u_1
  =
  F_1(u)
  =
  (\varrho_1 u_1^2+\varrho_2 u_2^2)u_1,\\
  (\Box +1)u_2
  =
  F_2(u)
  =
  (\varrho_3 u_1^2+\varrho_4 u_2^2)u_2
 \end{array}
 \right.
\end{align}
with $\varrho_1,\ldots,\varrho_4\in\R$, $\varrho_2\varrho_3>0$ 
and sufficiently small data size
$\|u(0)\|_{H^{4,1}}+\|\pa_t u(0)\|_{H^{3,1}}=\eps$.
For this system, we have
\begin{align*}
 \widetilde{F}(Y) 
 =
 \begin{pmatrix}
  \widetilde{F}_1(Y)\\
  \widetilde{F}_2(Y) 
 \end{pmatrix}
 =
 \begin{pmatrix}
  3\varrho_1 |Y_1|^2 Y_1 
  + 2\varrho_2 |Y_2|^2 Y_1 + \varrho_2 Y_2^2 \conj{Y_1}\\
  3\varrho_4 |Y_2|^2 Y_2 
  + 2\varrho_3 |Y_1|^2 Y_2 + \varrho_3 Y_1^2 \conj{Y_2}
 \end{pmatrix}
\end{align*}
so that
\begin{align*}
 \im \big( AY \cdot \widetilde{F}(Y) \big)
 =
 \im \left(
 \varrho_3 |\varrho_2| Y_1^2 \conj{Y_2}^2 
 +\varrho_2|\varrho_3|Y_2^2 \conj{Y_1}^2\right)
 =
 0
\end{align*}
with $A=\text{diag}\left(|\varrho_3|,|\varrho_2|\right)$.
Therefore we can apply Theorem~\ref{thm1} to see that
there exists a global small solution 
$u(t)\in C^0([0,\infty); H^{4,1}) \cap C^1 ([0,\infty); H^{3,1})$
to \eqref{sysexm1} and it satisfies the time-decay estimate \eqref{thm1-decay}.
\end{exm}

\begin{rmk}
Note that if $\varrho_1, \ldots, \varrho_4 =1$ then 
the system \eqref{sysexm1} is reduced to \eqref{nlkg-comp}
for the complex-valued unknown function $U$
via the relation $U=u_1 + iu_2$.
\end{rmk}

Now we perform a reduction of the problem
along the idea of \cite{Ha2008}
(see also \cite{Ha2009}--\cite{Ha2013}).
We define a new dependent variable and initial data
\begin{align*}
 v = \frac{1}{2} \left(u+i\jb{i\pa_x}^{-1}\pa_t u\right),
 \qquad
 v^\circ = \frac{1}{2} \left(f + i\jb{i\pa_x}^{-1} g\right)
\end{align*}
where $v=(v_j)_{1\le j\le N}$
and $v^\circ = (v_j^\circ)_{1\le j\le N}$.
Since $u$ consists of real-valued functions,
the system \eqref{nlkg}--\eqref{data} can be written as
\begin{align}\label{nlkg-var}
 \left\{
 \begin{array}{l}
 \L v= \G (v), \qquad (t,x)\in \R \times \R,\\
 v(0,x) = v^\circ (x),\qquad x\in\R,
 \end{array}
 \right.
\end{align}
where $\L = \pa_t + i\jb{i\pa_x}$ and the nonlinear term
\begin{align*}
 \G(v) = (\G_j(v))_{1\le j\le N}= 4i\jb{i\pa_x}^{-1} F(\re v)
 =
 \frac{i}{2}\ipax^{-1} F\left(v+\conj{v}\right)
\end{align*}
with the notations $\re v = (\re v_n)_{1\le n\le N}$
and $\conj v = (\conj{v_n})_{1\le n\le N}$.
Then it suffices to prove the following theorem.

\begin{thm}\label{thm2}
Let $v^\circ \in H^{4,1}$ and $\left\|v^\circ\right\|_{H^{4,1}} = \eps$.
Assume there exists an $N\times N$ positive Hermitian matrix $A$ such that
the condition \eqref{struc-condi} holds for all $Y\in\C^N$.
Then there exists $\eps_0>0$ such that for all $\eps \in (0, \eps_0)$,
the initial value problem \eqref{nlkg-var} admits a unique global solution
\begin{align*}
 v(t)\in C^0([0,\infty); H^{4,1})
\end{align*}
satisfying the time-decay estimate
\begin{align}\label{vdecay}
 \left\| v(t)\right\|_{H_\infty^1}
 \le
 C\eps (1+t)^{-1/2}
\end{align}
for all $t\ge 0$
with some positive constant $C$ which does not depend on $\eps$.
\end{thm}

We note that the solution $u$ for \eqref{nlkg}--\eqref{data} is represented by
$u=2\re v$, so Theorem~\ref{thm1} follows immediately
from Theorem~\ref{thm2}.

The rest of this paper is organized as follows. 
In Section~\ref{sec_pre}, we introduce some new operators
to handle the problem \eqref{nlkg-var}
and using them we perform a decomposition
of the free Klein-Gordon evolution group.
Section~\ref{sec_lem} is devoted to obtain
some lemmas involved in the proof of the main result. 
After that, we prove Theorem~\ref{thm2} in Section~\ref{sec_proof}
from which Theorem~\ref{thm1} follows immediately.
In what follows, all non-negative constants will be denoted by $C$
which may vary from line to line unless otherwise specified.

\section{Preliminaries} \label{sec_pre}
In this section, we carry out the decomposition of 
the free Klein-Gordon evolution group
$\expm = \IF e^{-it\jb{\xi}}\F$ into the main and remainder parts
along the lines of \cite{Ha2008}, \cite{Ha2010} and \cite{Ha2012}.
First we denote the dilation operator $\D_\omega$ by
\begin{align*}
 \D_\omega \phi = |\omega|^{-1/2}\phi(x\omega^{-1})
\end{align*}
so that $(\D_\omega)^{-1} = \D_{\omega^{-1}}$.
And we define the multiplication factor
\begin{align*}
 \M(t) = e^{-it\jb{ix}\theta(x)}
\end{align*}
where $\theta(x) =1$ if $|x|<1$ and $\theta(x) = 0$ if $|x|\ge 1$.
We also define the operator $\B$ by
\begin{align*}
 \B\phi = e^{-\pi i/4}\jb{ix}^{-3/2}\theta(x) \phi\left(x\jb{ix}^{-1}\right)
\end{align*}
and then the inverse operator $\IB$ acts on the function $\phi(x)$
defined on $(-1,1)$ as
\begin{align*}
 \IB\phi = e^{\pi i/4} \jb{\xi}^{-3/2} \phi\left(\xi \jb{\xi}^{-1}\right)
\end{align*}
for all $\xi \in \R$. 
Now we introduce the operators $\V$ and $\W$ as follows:
\begin{align*}
 &\V(t) = \IB \conj \M(t)\ID_t \IF e^{-it\jb{\xi}},\\
 &\W(t)= (1-\theta)\ID_t \IF e^{-it\jb{\xi}}.
\end{align*}
Noting that
\begin{align*}
 &\D_t \M(t)\B\V(t)=\D_t \M(t) \B\IB\conj\M(t)\ID_t\IF e^{-it\jb{\xi}}
 =\theta(x/t)\IF e^{-it\jb{\xi}},\\
 &\D_t \W(t)=\D_t(1-\theta)\ID_t\IF e^{-it\jb{\xi}}
 =(1-\theta(x/t))\IF e^{-it\jb{\xi}},
\end{align*}
we obtain the representation for the free Klein-Gordon evolution group:
\begin{align}\label{decompKG}
 e^{-it\jb{i\pa_x}} \IF
 &:=
 \IF e^{-it\jb{\xi}}
 =
 \D_t \M(t) \B\V(t) + \D_t \W(t)\nonumber\\
 &=
 \D_t \M(t)\B + \D_t \M(t) \B(\V(t)-1)+\D_t \W(t).
\end{align}
As we shall see in Section~\ref{sec_proof},
the second and the third term of \eqref{decompKG}
can be regarded as remainder parts,
while the first term of \eqref{decompKG} plays a role as a main term.
Now we let
\begin{align*}
 &\IV(t) = e^{it\jb{\xi}}\F\D_t\M(t)\B,\\
 &\IW(t)= e^{it\jb{\xi}}\F\D_t(1-\theta)
\end{align*}
so that we get
\begin{align*}
 &\IV(t)\IB\conj\M(t)\ID_t
 =e^{it\jb{\xi}}\F\D_t\M(t)\B\IB\conj\M(t)\ID_t
 =e^{it\jb{\xi}}\F\theta(x/t),\\
 &\IW(t)\ID_t
 =
 e^{it\jb{\xi}}\F\D_t(1-\theta)\ID_t
 =e^{it\jb{\xi}}\F(1-\theta(x/t)).
\end{align*}
Then similarly we have the following representation:
\begin{align*}
 \F e^{it\jb{i\pa_x}}
 &:=e^{it\jb{\xi}}\F
 =
 \IV(t)\IB\conj\M(t)\ID_t + \IW(t)\ID_t\\
 &=\IB\conj\M(t)\ID_t + (\IV(t)-1)\IB\conj\M(t)\ID_t + \IW(t)\ID_t.
\end{align*}
Now we define the operator $\J$ by
\begin{align*}
 \J = \ipax\expm x \expp = \IF \jb{\xi}\expmxi i\pa_\xi \exppxi \F
 =\ipax x + it\pa_x.
\end{align*}
The operator $\J$ was frequently used in the previous works
\cite{Ha2008,Ha2010,Ha2012,Ha2013}
to deal with the nonlinear Klein-Gordon equations.
Since $\J$ is not a purely differential operator, it is not that easy
to calculate the action of $\J$ on the nonlinearity in \eqref{nlkg-var}.
So instead we employ the operator
\begin{align*}
 \P = t\pa_x + x\pa_t
\end{align*}
which is closely related to $\J$ via the identities
\begin{align*}
 \J = i\P - ix\L - \ipax^{-1}\pa_x,
 \qquad
 \P = \L x - i \J.
\end{align*}
We close this section by introducing some commutation relations
which can be easily shown through direct calculations:
\begin{align*}
 &[x, \ipax^\beta]  = \beta \ipax^{\beta-2}\pa_x,\qquad
 & &[\L, \J] =0,\\ 
 &[\L, \P] =-i\ipax^{-1}\pa_x \L,\qquad
 & &[\P, \ipax^\beta] = \beta \ipax^{\beta-2}\pa_x\pa_t
\end{align*}
valid for any real number $\beta$.

\section{Lemmas} \label{sec_lem}
In this section, we introduce several lemmas
which will be used in the proof of the main result.
In what follows, we will derive large time asymptotics for the free
Klein-Gordon evolution group.
\begin{lem}\label{lem1}
The estimates
\begin{align*}
 &\left\|\V(t)\phi\right\|_{H^{1,1-\eta}} \le C\left\|\phi\right\|_{H^{1,4}},\\
 &\left\|\jb{\xi}^{3/2}(\V(t)-1)\phi\right\|_{L^\infty} 
 \le Ct^{-1/4}\left\|\phi\right\|_{H^{1,3}},\\
 &\left\|\W(t)\phi\right\|_{L^r} \le Ct^{-1/2}\left\|\phi\right\|_{H^{1,3}}
\end{align*}
and
\begin{align*}
 \left\|(\IV(t)-1)\phi\right\|_{L^\infty}
 \le
 Ct^{-1/4}\left\|\phi\right\|_{H^{1,\frac{3}{4}+\eta}}
\end{align*}
hold for $t\ge 1$, $r\in[2,\infty]$ and $\eta\in(0,1)$
provided that the right hand sides are finite.
\end{lem}

We omit the proof of this lemma because it is exactly the same as
that of the previous works (\cite{Ha2008}, \cite{Ha2012}, etc.). 
Next, we introduce a time-decay estimate in terms of the operator $\J$
whose proof can be found in \cite{Ha2008} or \cite{Ha2010}
(see also \cite{Ha2012}).

\begin{lem}\label{lem2}
 The estimate
 \begin{align*}
 \left\|\phi\right\|_{L^\infty}
 \le
 C\jb{t}^{-1/2}\|\phi\|_{H^{3/2}}^{1/2}
 \left(\|\phi\|_{H^{3/2}}^{1/2}+\|\J\phi\|_{H^{1/2}}^{1/2}\right)
 \end{align*}
 is valid for all $t\ge0$, provided that the right-hand side is finite.
\end{lem}

In the next lemma we obtain large-time asymptotics 
for the nonlinear term $\G$ in the equation \eqref{nlkg-var}.
\begin{lem}\label{lem3}
 Let $\phi = (\phi_j)_{1\le j\le N}\in H^{4,1}$. 
 Then we have the following expression
 \begin{align*}
  &\F\expp\ipax\G_j\left(\expm\phi\right)\\
  &\qquad=
  it^{-1}\exppxi
  \sum_{k,l,m=1}^N
  \sum_{\nu=1}^8
  \mu_\nu C_{j,k,l,m}
  \D_{\omega_\nu} e^{-it\omega_\nu\jb{\xi}}\jb{\xi}^3
  \prod_{\lambda=k,l,m}\left(
  \conj{\hat{\phi_\lambda}}^{1-\alpha_{\lambda_\nu}}
  \hat{\phi_\lambda}^{\alpha_{\lambda_\nu}}\right)
  +R_j(t),
 \end{align*}
 where the real constants $C_{j,k,l,m}$ are given in \eqref{Fjform}
 and the remainder $R_j$ satisfies the estimate
 \begin{align}\label{lem3-remainder}
 \left\|R_j(t)\right\|_{L^\infty}
 \le
 Ct^{-5/4}\left\|\phi\right\|_{H^{4,1}}^3
 \end{align}
 for $t\ge 1$, $j=1,\ldots,N$,
 Here $\mu_\nu$, $\alpha_{\lambda_\nu}$, $\omega_\nu$ are determined by
 \begin{align*}
  &(\alpha_{k_1}, \alpha_{l_1}, \alpha_{m_1})=(1,1,1),
  \qquad\omega_1 =+3,\qquad\mu_1=-i/2,\\
  &(\alpha_{k_2}, \alpha_{l_2}, \alpha_{m_2})=(1,1,0),
  \qquad\omega_2 =+1,\qquad\mu_2=+1/2,\\
  &(\alpha_{k_3}, \alpha_{l_3}, \alpha_{m_3})=(1,0,1),
  \qquad\omega_3 =+1,\qquad\mu_3=+1/2,\\
  &(\alpha_{k_4}, \alpha_{l_4}, \alpha_{m_4})=(1,0,0),
  \qquad\omega_4 =-1,\qquad\mu_4=+i/2,\\
  &(\alpha_{k_5}, \alpha_{l_5}, \alpha_{m_5})=(0,1,1),
  \qquad\omega_5 =+1,\qquad\mu_5=+1/2,\\
  &(\alpha_{k_6}, \alpha_{l_6}, \alpha_{m_6})=(0,1,0),
  \qquad\omega_6 =-1,\qquad\mu_6=+i/2,\\
  &(\alpha_{k_7}, \alpha_{l_7}, \alpha_{m_7})=(0,0,1),
  \qquad\omega_7 =-1,\qquad\mu_7=+i/2,\\
  &(\alpha_{k_8}, \alpha_{l_8}, \alpha_{m_8})=(0,0,0),
  \qquad\omega_8 =-3,\qquad\mu_8=-1/2,
 \end{align*}
 for any $k,l,m\in\{1,\ldots,N\}$.
\end{lem}

\begin{proof}
First we prove the following representation for $t\ge 1$
\begin{align}\label{lem3-expression}
 &\F\expp\left(\conj{v_k}(t)^{1-\alpha_k}v_k(t)^{\alpha_k}
 \conj{v_l}(t)^{1-\alpha_l}v_l(t)^{\alpha_l}
 \conj{v_m}(t)^{1-\alpha_m}v_m(t)^{\alpha_m}\right)\nonumber\\
 &\qquad=
 -e^{-\frac{\pi}{2} i(\alpha_k+\alpha_l+\alpha_m)}t^{-1}\exppxi
 \D_\omega e^{-it\omega\jb{\xi}}\jb{\xi}^3
 \prod_{\lambda=k,l,m}
 \left(
 \conj{\varphi_\lambda}^{1-\alpha_\lambda}
 \varphi_\lambda^{\alpha_\lambda}
 \right)
 +R_{klm}(t)
\end{align}
where $\varphi_\lambda= \varphi_\lambda(t)= \F\expp v_\lambda(t)$
for each $\lambda=k,l,m$ and
$\alpha_k, \alpha_l, \alpha_m \in \{0,1\}$,
$\omega=2(\alpha_k+\alpha_l+\alpha_m)-3$
with
$\alpha_k+\alpha_l+\alpha_m\neq 3/2$
and the remainder satisfies
\begin{align*}
 \left\|R_{klm}(t)\right\|_{L^\infty}
 \le
 Ct^{-5/4} \left\|\varphi\right\|_{H^{1,4}}^3
\end{align*}
for all $k,l,m\in\{1,\ldots,N\}$.
Here we write $\varphi = (\varphi_1, \ldots, \varphi_N)$.
We introduce a new operator $\Q(t)$ which is defined by
\begin{align*}
 \Q(t) := \B\V(t) + \W(t) = \conj\M(t)\ID_t\IF\expmxi.
\end{align*}
With this operator, we can rewrite the representation for the free
Klein-Gordon evolution group as
\begin{align*}
 &\expm\IF = \IF\expmxi = \D_t\M(t)\Q(t),\\
 &\F\expp = \exppxi\F = \IQ(t)\conj\M(t)\ID_t,
\end{align*}
where
\begin{align*}
 \IQ(t) = \IV(t)\IB + \IW(t) = \exppxi\F\D_t\M(t).
\end{align*}
Since $\D_{\omega t} = \D_\omega \D_t$,
$\ID_\omega = \D_{\omega^{-1}}$
and
$\F\D_{\omega^{-1}} = \D_\omega\F$,
we have
\begin{align}\label{IQM}
 \IQ(t)\M^{\omega-1}
 =\exppxi\F\D_t\M^\omega(t)
 =\exppxi\D_\omega\F\D_{\omega t}\M^\omega(t)
 =\exppxi\D_\omega e^{-it\omega\jb{\xi}}\IQ(\omega t)
\end{align}
when $\omega\neq 0$.
Now we put 
\begin{align*}
 v_\lambda(t)
 =
 \IF\expmxi\varphi_\lambda
 =
 \D_t\M(t)\Q(t)\varphi_\lambda
\end{align*}
for $\lambda=k,l,m$ respectively.
Then taking $\omega = 2(\alpha_k +\alpha_l +\alpha_m)-3$, we get
\begin{align}
 &\exppxi\F
 \prod_{\lambda=k,l,m}\left(
 \conj{v_\lambda}^{1-\alpha_\lambda} v_\lambda^{\alpha_\lambda}\right)
 \nonumber\\
 &\quad=
 \IQ(t)\conj\M(t)\ID_t
 \prod_{\lambda=k,l,m}
 \left(\left(\conj{\D_t\M(t)\Q(t)\varphi_\lambda}\right)^{1-\alpha_\lambda}
 \Big(\D_t\M(t)\Q(t)\varphi_\lambda\Big)^{\alpha_\lambda}
 \right)\nonumber\\
 &\quad=
 \IQ(t)\M^{-1}(t)t^{1/2}\prod_{\lambda=k,l,m}\left(
 t^{-\frac{1-\alpha_\lambda}{2}}t^{-\frac{\alpha_\lambda}{2}}
 \M^{-(1-\alpha_\lambda)}(t)\M^{\alpha_\lambda}(t)
 \left(\conj{\Q(t)\varphi_\lambda}\right)^{1-\alpha_\lambda}
 \Big(\Q(t)\varphi_\lambda\Big)^{\alpha_\lambda}
 \right)\nonumber\\
 &\quad=
 t^{-1}\IQ(t)\M^{2(\alpha_k+\alpha_l+\alpha_m)-4}(t)
 \prod_{\lambda=k,l,m}\left(
 \left(\conj{\Q(t)\varphi_\lambda}\right)^{1-\alpha_\lambda}
 \Big(\Q(t)\varphi_\lambda\Big)^{\alpha_\lambda}
 \right)\nonumber\\
 &\quad=
 t^{-1}\exppxi\D_\omega e^{-it\omega\jb{\xi}}\IQ(\omega t)
 \prod_{\lambda=k,l,m}\left(
 \left(\conj{\Q(t)\varphi_\lambda}\right)^{1-\alpha_\lambda}
 \Big(\Q(t)\varphi_\lambda\Big)^{\alpha_\lambda}
 \right)
 \label{lem3-3}
\end{align}
where we used \eqref{IQM} for the last equality.
Since $\Q(t) = \B\V(t)$ for $|x|\le 1$ 
and $\Q(t)=\W(t)$ for $|x|\ge 1$, we have
\begin{align*}
 &\prod_{\lambda=k,l,m}\left(
 \left(\conj{\Q(t)\varphi_\lambda}\right)^{1-\alpha_\lambda}
 \Big(\Q(t)\varphi_\lambda\Big)^{\alpha_\lambda}
 \right)\\
 &\qquad=
 \prod_{\lambda=k,l,m}\left(
 \left(\conj{\B\V(t)\varphi_\lambda}\right)^{1-\alpha_\lambda}
 \Big(\B\V(t)\varphi_\lambda\Big)^{\alpha_\lambda}
 \right)
 +
 \prod_{\lambda=k,l,m}\left(
 \left(\conj{\W(t)\varphi_\lambda}\right)^{1-\alpha_\lambda}
 \Big(\W(t)\varphi_\lambda\Big)^{\alpha_\lambda}
 \right)
\end{align*}
for all $x\in\R$. In the same manner we obtain
\begin{align}\label{lem3-4}
 &\IQ(\omega t)\prod_{\lambda=k,l,m}\left(
 \left(\conj{\Q(t)\varphi_\lambda}\right)^{1-\alpha_\lambda}
 \Big(\Q(t)\varphi_\lambda\Big)^{\alpha_\lambda}
 \right)\nonumber\\
 &\qquad=
 \IV(\omega t)\IB
 \prod_{\lambda=k,l,m}\left(
 \left(\conj{\B\V(t)\varphi_\lambda}\right)^{1-\alpha_\lambda}
 \Big(\B\V(t)\varphi_\lambda\Big)^{\alpha_\lambda}
 \right)\nonumber\\
 &\qquad\qquad+
 \IW(\omega t)
 \prod_{\lambda=k,l,m}\left(
 \left(\conj{\W(t)\varphi_\lambda}\right)^{1-\alpha_\lambda}
 \Big(\W(t)\varphi_\lambda\Big)^{\alpha_\lambda}
 \right).
\end{align}
Plugging \eqref{lem3-4} into \eqref{lem3-3} and applying 
the following identity
(which follows immediately from the definitions of 
the operators $\B$ and $\IB$)
\begin{align*}
 \IB \prod_{\lambda=k,l,m}\left(
 \big(\conj{\B\chi_\lambda}\big)^{1-\alpha_\lambda}
 \big(\B\chi_\lambda\big)^{\alpha_\lambda}
 \right)
 =
 -e^{-\frac{\pi i}{2} (\alpha_k + \alpha_l + \alpha_m)}
 \jb{\xi}^3
 \prod_{\lambda=k,l,m}\left(
 \big(\conj{\chi_\lambda}(\xi)\big)^{1-\alpha_\lambda}
 \big(\chi_\lambda(\xi)\big)^{\alpha_\lambda}
 \right),
\end{align*}
we obtain
\begin{align}\label{eitxi-46}
 &\exppxi\F
 \prod_{\lambda=k,l,m}
 \left(\conj{v_\lambda}^{1-\alpha_\lambda} 
 v_\lambda^{\alpha_\lambda}\right)
 \nonumber\\
 &\qquad=
 -e^{-\frac{\pi i}{2}(\alpha_k +\alpha_l +\alpha_m)} t^{-1}\exppxi
 \D_\omega e^{-it\omega\jb{\xi}} \IV(\omega t)
 \jb{\xi}^3 \prod_{\lambda=k,l,m}\left(
 \left(\conj{\V(t)\varphi_\lambda}\right)^{1-\alpha_\lambda}
 \Big(\V(t)\varphi_\lambda\Big)^{\alpha_\lambda}
 \right)\nonumber\\
 &\qquad\qquad+
 t^{-1}\exppxi\D_\omega e^{-it\omega\jb{\xi}}
 \IW(\omega t)
 \prod_{\lambda=k,l,m}\left(
 \left(\conj{\W(t)\varphi_\lambda}\right)^{1-\alpha_\lambda}
 \Big(\W(t)\varphi_\lambda\Big)^{\alpha_\lambda}
 \right)\nonumber\\
 &\qquad=
 -e^{-\frac{\pi i}{2}(\alpha_k +\alpha_l +\alpha_m)} t^{-1}
 \exppxi\D_\omega e^{-it\omega\jb{\xi}}
 \jb{\xi}^3
 \prod_{\lambda=k,l,m}\left(
 \conj{\varphi_\lambda}^{1-\alpha_\lambda}
 \varphi_\lambda^{\alpha_\lambda}
 \right) + R_{klm}(t),
\end{align}
where the remainder $R_{klm} = R_{klm}^{\rm I} + R_{klm}^{\rm II}$
is given by
\begin{align*}
 &R_{klm}^{\rm I}
 =
 -e^{-\frac{\pi i}{2}(\alpha_k +\alpha_l +\alpha_m)} t^{-1}\exppxi
 \D_\omega e^{-it\omega\jb{\xi}}\\
 &\qquad\qquad\times\bigg[\IV(\omega t)
 \jb{\xi}^3 \prod_{\lambda=k,l,m}\left(
 \left(\conj{\V(t)\varphi_\lambda}\right)^{1-\alpha_\lambda}
 \Big(\V(t)\varphi_\lambda\Big)^{\alpha_\lambda}
 \right)
 -
 \jb{\xi}^3
 \prod_{\lambda=k,l,m}\left(
 \conj{\varphi_\lambda}^{1-\alpha_\lambda}
 \varphi_\lambda^{\alpha_\lambda}
 \right)\bigg],\\
 &R_{klm}^{\rm II}
 =
 t^{-1}\exppxi\D_\omega e^{-it\omega\jb{\xi}}
 \IW(\omega t)
 \prod_{\lambda=k,l,m}\left(
 \left(\conj{\W(t)\varphi_\lambda}\right)^{1-\alpha_\lambda}
 \Big(\W(t)\varphi_\lambda\Big)^{\alpha_\lambda}
 \right).
\end{align*}
Now we estimate $R_{klm}$.
By Lemma~\ref{lem1} and the standard Sobolev embedding, we evaluate
\begin{align*}
 &\left\|
 (\IV(\omega t)-1)
 \jb{\xi}^3 \prod_{\lambda=k,l,m}\left(
 \left(\conj{\V(t)\varphi_\lambda}\right)^{1-\alpha_\lambda}
 \Big(\V(t)\varphi_\lambda\Big)^{\alpha_\lambda}
 \right)\right\|_{L^\infty}\\
 &\qquad\le
  Ct^{-1/4}\left\|
 \jb{\xi}^{\frac{15}{4}+\eta}
 \prod_{\lambda=k,l,m}\left(
 \left(\conj{\V(t)\varphi_\lambda}\right)^{1-\alpha_\lambda}
 \Big(\V(t)\varphi_\lambda\Big)^{\alpha_\lambda}
 \right)\right\|_{H^1}\\
 &\qquad\le
 Ct^{-1/4}\left\|
 \jb{\xi}^{\frac{11}{8}+\eta} \V(t)\varphi \right\|_{L^\infty}^2
 \left\|\V(t)\varphi\right\|_{H^{1,1-\eta}}\\
 &\qquad\le
 Ct^{-1/4}\left\|
 \jb{\xi}^{3/2} \V(t)\varphi \right\|_{L^\infty}^2
 \left\|\varphi\right\|_{H^{1,4}}\\
 &\qquad\le
 Ct^{-1/4}\left(
 \left\|\jb{\xi}^{3/2}(\V(t)-1)\varphi\right\|_{L^\infty}^2 +
 \left\|\jb{\xi}^{3/2}\varphi\right\|_{L^\infty}^2\right)
 \left\|\varphi\right\|_{H^{1,4}}\\
 &\qquad\le
 Ct^{-1/4} \left(t^{-1/2} \left\|\varphi\right\|_{H^{1,3}}^2
 +\left\|\varphi\right\|_{H^{1,3/2}}^2\right)
 \left\|\varphi\right\|_{H^{1,4}}\\
 &\qquad\le
 Ct^{-1/4} 
 \left\|\varphi\right\|_{H^{1,3}}^2 \left\|\varphi\right\|_{H^{1,4}}
\end{align*}
with $0<\eta<1/8$.
Also in view of the relation
\begin{align*}
 \chi_1 \chi_2 \chi_3 - 
 \widetilde \chi_1\widetilde \chi_2\widetilde \chi_3
 &=
 \frac{1}{3}(\chi_1 - \widetilde\chi_1)\left(
 (\chi_2 - \widetilde\chi_2)
 (\chi_3-\widetilde\chi_3)
 +\chi_2\widetilde\chi_3 + 2\chi_3\widetilde\chi_2\right)\\
 &\qquad+
  \frac{1}{3}(\chi_2 - \widetilde\chi_2)\left(
 (\chi_3-\widetilde\chi_3)(\chi_1 - \widetilde\chi_1)
 +\chi_3\widetilde\chi_1 + 2\chi_1\widetilde\chi_3\right)\\
  &\qquad+
  \frac{1}{3}(\chi_3 - \widetilde\chi_3)\left(
 (\chi_1 - \widetilde\chi_1)(\chi_2-\widetilde\chi_2)
 +\chi_1\widetilde\chi_2 + 2\chi_2\widetilde\chi_1\right)
\end{align*}
and by Lemma~\ref{lem1}, we have
\begin{align*}
 &\left\|
 \jb{\xi}^3\left( \prod_{\lambda=k,l,m}\left(
 \left(\conj{\V(t)\varphi_\lambda}\right)^{1-\alpha_\lambda}
 \Big(\V(t)\varphi_\lambda\Big)^{\alpha_\lambda}
 \right)
 -
 \prod_{\lambda=k,l,m}\left(
 \conj{\varphi_\lambda}^{1-\alpha_\lambda}
 \varphi_\lambda^{\alpha_\lambda}
 \right)\right)
 \right\|_{L^\infty}\\
 &\qquad\le
 C\left(
 \left\|\jb{\xi}^{3/2}(\V(t)-1)\varphi\right\|_{L^\infty}^2
 +\left\|\jb{\xi}^{2+\eta}\varphi\right\|_{L^\infty}
 \left\|\jb{\xi}^{1-\eta}\V\varphi\right\|_{L^\infty}\right)
 \big\|(\V(t)-1)\varphi\big\|_{L^\infty}\\
 &\qquad\le
 C\left(t^{-1/2}\left\|\varphi\right\|_{H^{1,3}}^2
 +\left\|\varphi\right\|_{H^{1,2+\eta}}
 \left\|\V\varphi\right\|_{H^{1,1-\eta}}\right)
 t^{-1/4}\left\|\varphi\right\|_{H^{1,3}}\\
 &\qquad\le
 Ct^{-1/4}
 \left\|\varphi\right\|_{H^{1,3}}^2\left\|\varphi\right\|_{H^{1,4}}
\end{align*}
with $0<\eta<1$.
Since $\F\D_{\omega t}=\F\D_\omega\D_t
=\D_{\omega^{-1}}\D_{t^{-1}} \F$,
from the definition of $\IW$ and the third estimate of Lemma~\ref{lem1}, 
we get
\begin{align*}
 &\left\|
 \IW(\omega t)
 \prod_{\lambda=k,l,m}\left(
 \left(\conj{\W(t)\varphi_\lambda}\right)^{1-\alpha_\lambda}
 \Big(\W(t)\varphi_\lambda\Big)^{\alpha_\lambda}
 \right)\right\|_{L^\infty}\\
 &\qquad\le
 \left\|
 e^{i\omega t\jb{\xi}} \F\D_{\omega t}
 \prod_{\lambda=k,l,m}\left(
 \left(\conj{\W(t)\varphi_\lambda}\right)^{1-\alpha_\lambda}
 \Big(\W(t)\varphi_\lambda\Big)^{\alpha_\lambda}
 \right)
 \right\|_{L^\infty}\\
 &\qquad\le
 Ct^{1/2} \left\|
 \F\prod_{\lambda=k,l,m}\left(
 \left(\conj{\W(t)\varphi_\lambda}\right)^{1-\alpha_\lambda}
 \Big(\W(t)\varphi_\lambda\Big)^{\alpha_\lambda}
 \right)
 \right\|_{L^\infty}\\
 &\qquad\le
 Ct^{1/2} \left\|
 \prod_{\lambda=k,l,m}\left(
 \left(\conj{\W(t)\varphi_\lambda}\right)^{1-\alpha_\lambda}
 \Big(\W(t)\varphi_\lambda\Big)^{\alpha_\lambda}
 \right)
 \right\|_{L^1}\\
 &\qquad\le
 Ct^{1/2}\left\|
 \W(t)\varphi
 \right\|_{L^3}^3
 \le
 Ct^{-1} \left\|\varphi\right\|_{H^{1,3}}^3.
\end{align*}
Combining all together, we find
\begin{align}\label{Rklm47}
 \|R_{klm}(t)\|_{L^\infty}
 &\le
 Ct^{-1}
 \left\|
 (\IV(\omega t)-1)
 \jb{\xi}^3 \prod_{\lambda=k,l,m}\left(
 \left(\conj{\V(t)\varphi_\lambda}\right)^{1-\alpha_\lambda}
 \Big(\V(t)\varphi_\lambda\Big)^{\alpha_\lambda}
 \right)\right\|_{L^\infty}\nonumber\\
 &\quad+
 Ct^{-1}
 \left\|
 \jb{\xi}^3\left( \prod_{\lambda=k,l,m}\left(
 \left(\conj{\V(t)\varphi_\lambda}\right)^{1-\alpha_\lambda}
 \Big(\V(t)\varphi_\lambda\Big)^{\alpha_\lambda}
 \right)
 -
 \prod_{\lambda=k,l,m}\left(
 \conj{\varphi_\lambda}^{1-\alpha_\lambda}
 \varphi_\lambda^{\alpha_\lambda}
 \right)\right)
 \right\|_{L^\infty}\nonumber\\
 &\quad+
 Ct^{-1}
 \left\|
 \IW(\omega t)
 \prod_{\lambda=k,l,m}\left(
 \left(\conj{\W(t)\varphi_\lambda}\right)^{1-\alpha_\lambda}
 \Big(\W(t)\varphi_\lambda\Big)^{\alpha_\lambda}
 \right)\right\|_{L^\infty}\nonumber\\
 &\le
 Ct^{-5/4}\left\|\varphi\right\|_{H^{1,4}}^3.
\end{align}
Therefore by \eqref{eitxi-46} and \eqref{Rklm47},
we arrive at the expression \eqref{lem3-expression}.
Now we take $v_j=\expm\phi_j$ in \eqref{lem3-expression}
so that $\varphi_j = \F\phi_j$.
Finally we obtain from \eqref{lem3-expression} that
\begin{align*}
 &\F\expp\ipax\G_j(v)
 =
 \frac{i}{2}\F\expp F_j(v+\conj{v})\\
 &\qquad=
 \frac{i}{2}\F\expp \sum_{k,l,m=1}^N C_{j,k,l,m}
 (v_k + \conj{v_k})(v_l+\conj{v_l})(v_m+\conj{v_m})\\
 &\qquad=
 it^{-1}\exppxi
 \sum_{k,l,m=1}^N
 \sum_{\nu=1}^8
 \mu_\nu C_{j,k,l,m}
 \D_{\omega_\nu} e^{-it\omega_\nu\jb{\xi}}\jb{\xi}^3
 \prod_{\lambda=k,l,m}\left(
 \conj{\hat{\phi_\lambda}}^{1-\alpha_{\lambda_\nu}}
 \hat{\phi_\lambda}^{\alpha_{\lambda_\nu}}\right)
 +R_j(t)
\end{align*}
with appropriate constants $\mu_\nu$ given above
and the remainder $R_j$ satisfying \eqref{lem3-remainder}.
Lemma~\ref{lem3} is proved.
\end{proof}

Now we will derive a system of ordinary differential equations
for the new variable $\psi(t)$ which is deeply related to our problem.
The following lemma plays an essential role in the proof of Theorem~\ref{thm2}.

\begin{lem}\label{lem4}
Let us define $\psi = \psi(t,\xi) =(\psi_j (t,\xi))_{1\le j\le N}$ by
$\psi_j = \jb{\xi}^\kappa\F\expp v_j$
where $\kappa\ge 1$.
Then $\psi$ satisfies the following system of 
ordinary differential equations with the parameter $\xi\in\R$
\begin{align}\label{odesystem}
 \pa_t\psi (t) = \frac{1}{2}it^{-1}\jb{\xi}^{2-2\kappa}
 \widetilde{F}(\psi(t))+S(t)+R(t)
\end{align}
for $t\ge 1$, where $\widetilde{F}$ is given in \eqref{tildeF}, 
the non-resonant term $S(t) = ((S_j(t))_{1\le j\le N}$ is
\begin{align*}
 S_j(t)
 =
 it^{-1}\jb{\xi}^{\kappa-1}\sum_{k,l,m=1}^N\sum_{\nu=1,4,6,7,8}
 C_{j,k,l,m}^\nu \exppxi \D_{\omega_\nu}
 e^{-it\omega_\nu \jb{\xi}}
 \jb{\xi}^{3-3\kappa}\prod_{\lambda=k,l,m}\left(
 \conj{\psi_\lambda}^{1-\alpha_{\lambda_\nu}}
 \psi_\lambda^{\alpha_{\lambda_\nu}}\right)
\end{align*}
with some constants $C_{j,k,l,m}^\nu\in\C$,
$\alpha_{\lambda_\nu}, \omega_\nu$ given in the previous lemma
and the remainder $R(t)=(R_j(t))_{1\le j\le N}$ satisfying
\begin{align*}
 \left\|R_j(t)\right\|_{L^\infty}
 \le 
 Ct^{-5/4}\left\|\psi\right\|_{H^{1,4-\kappa}}^3.
\end{align*}
\end{lem}

\begin{proof}
First we let $\psi^1 = \psi |_{\kappa =1}$, that is,
$\psi_j^1 = \jb{\xi}\F\expp v_j$. 
We multiply both sides of \eqref{nlkg-var} by
$\F\ipax\expp = \jb{\xi}\exppxi\F$.
Noting that
$v_j = \expm\IF\jb{\xi}^{-1}\psi_j^1$,
we have
\begin{align}\label{psipde}
 \pa_t \psi_j^1
 &=
 \jb{\xi}\exppxi(i\jb{\xi}+\pa_t)\F v_j
 =
 \jb{\xi} \exppxi \F\L v_j
 =
 \jb{\xi} \exppxi \F\G_j (v)\nonumber\\
 &=
 \F\expp\ipax\G_j(v)
 =
 \F\expp\ipax\G_j(\expm\IF\jb{\xi}^{-1}\psi^1).
\end{align}
Applying Lemma~\ref{lem3} to the right-hand side of \eqref{psipde},
we obtain
\begin{align}
 \pa_t\psi_j^1
 &=
 it^{-1} \sum_{k,l,m=1}^N\sum_{\nu=1}^8
 \mu_\nu C_{j,k,l,m}\exppxi\D_{\omega_\nu}e^{-it\omega_\nu\jb{\xi}}
 \prod_{\lambda=k,l,m}\left(
 (\conj{\psi_\lambda^1})^{1-\alpha_{\lambda_\nu}}
 {(\psi_\lambda^1)}^{\alpha_{\lambda_\nu}}\right)+R_j^1(t)
 \label{psipde2-1}
\end{align}
with $R_j^1$ satisfying
$\left\|R_j^1(t)\right\|_{L^\infty}\le
Ct^{-5/4}\left\|\IF\jb{\xi}^{-1}\psi^1\right\|_{H^{4,1}}^3$.
Since $\psi = \jb{\xi}^{\kappa-1}\psi^1$, replacing
$\psi_j^1$ by $\jb{\xi}^{-\kappa+1}\psi_j$ in \eqref{psipde2-1},
we get
\begin{align}
 \pa_t\psi_j
 &=
 it^{-1}\jb{\xi}^{\kappa-1}\sum_{k,l,m=1}^N\sum_{\nu=1}^8
 \mu_\nu C_{j,k,l,m}\exppxi\D_{\omega_\nu}e^{-it\omega_\nu\jb{\xi}}
 \jb{\xi}^{3-3\kappa}\prod_{\lambda=k,l,m}\left(
 \conj{\psi_\lambda}^{1-\alpha_{\lambda_\nu}}
 {\psi_\lambda}^{\alpha_{\lambda_\nu}}\right)+R_j(t)
 \label{psipde2}
\end{align}
where $R_j$ satisfies
\begin{align*}
 \left\|R_j(t)\right\|_{L^\infty}
 \le
 Ct^{-5/4}\left\|\IF\jb{\xi}^{-\kappa}\psi\right\|_{H^{4,1}}^3
 \le
 Ct^{-5/4}\left\|\psi\right\|_{H^{1,4-\kappa}}^3.
\end{align*}
Now we separate the right-hand side of \eqref{psipde2}
to resonant terms
(i.e. $\omega_\nu =1$) and non-resonant terms
(i.e. $\omega_\nu \neq 1$).
From Lemma~\ref{lem3}, we know that $\omega_\nu =1$
if and only if $\nu = 2,3$ or $5$.
Consequently with $S_j$ given above, we arrive at
\begin{align*}
 \pa_t \psi_j
 &=
 it^{-1}\jb{\xi}^{2-2\kappa} \sum_{k,l,m=1}^N \sum_{\nu=2,3,5}
 \mu_\nu C_{j,k,l,m}
 \prod_{\lambda=k,l,m}\left(
 \conj{\psi_\lambda}^{1-\alpha_{\lambda_\nu}}
 \psi_\lambda^{\alpha_{\lambda_\nu}}\right)
 +S_j (t) + R_j(t)\\
 &=
 \frac{1}{2}it^{-1}\jb{\xi}^{2-2\kappa}
 \sum_{k,l,m=1}^N C_{j,k,l,m}
 \Big(\psi_k \psi_l \conj{\psi_m}
 +\psi_k \conj{\psi_l} \psi_m
 +\conj{\psi_k}\psi_l \psi_m\Big)
 +S_j (t) + R_j(t)
\end{align*}
 for $j=1,\ldots,N$, which proves Lemma~\ref{lem4}.
\end{proof}

\section{Proof of Theorem~\ref{thm2}} \label{sec_proof}
In this section, we prove Theorem~\ref{thm2} along the idea of \cite{Ha2008}.
In \cite{Ha2008}, they chose a suitable phase function to remove
resonant terms appearing in a certain ODE similar to \eqref{odesystem},
and then estimated the variable $\jb{\xi}\F\expp v$ properly.
However, the method there seems not applicable directly to our problem
since we cannot choose such a suitable phase function for \eqref{odesystem}
(note that this difficulty prevent us 
getting some asymptotics for the solution).
So instead, here we use the method in \cite{Su2005-2} with slight variations
to estimate $\jb{\xi}\F\expp v$ (this kind of technique
was also used in \cite{Kim1}, \cite{Kim2}, \cite{Su2006}, etc.).\\

We introduce a function space
\begin{align*}
 X_T
 =
 \left\{
 \phi=(\phi_1,\ldots,\phi_N)\in C^0([0,T];L^2) : \|\phi\|_{X_T}<\infty \right\}
\end{align*}
where
\begin{align*}
 \left\|\phi\right\|_{X_T}
 =
 \sup_{t\in[0,T]}
 \left(\jb{t}^{-\gamma}\left\|\phi(t)\right\|_{H^4}
 +\jb{t}^{-\gamma}\left\|\J\phi(t)\right\|_{H^2}
 +\jb{t}^{-3\gamma}\left\|\J\phi(t)\right\|_{H^3}
 +\jb{t}^{1/2}\left\|\phi(t)\right\|_{H_\infty^1}\right)
\end{align*}
with $0<\gamma \ll 1$ small.\\

The local existence in the function space $X_T$ can be proved by
the standard contraction mapping principle.
Here we assume the following local existence theorem:
\begin{thm}\label{local-existence}
(Local Existence) Let $v^\circ \in H^{4,1}$ and
$\left\|v^\circ\right\|_{H^{4,1}}=\eps$.
Then there exist $\eps_0>0$ and $T>1$ such that for all $\eps\in(0,\eps_0)$,
the initial value problem \eqref{nlkg-var} admits a unique local solution
$v\in C^0([0,T];H^{4,1})$ with the estimate $\left\|v\right\|_{X_T}<\sqrt{\eps}$.
\end{thm}

Let us prove that the existence time $T$ can be extended to infinity.
We will show this by contradiction. 
We assume there exists a minimal time $T>0$ such that 
$\left\|v\right\|_{X_T}<\sqrt\eps$ does not hold,
that is, we have $\left\|v\right\|_{X_T}\le\sqrt\eps$.
In what follows we 
will prove $\left\|v\right\|_{X_T}\le C\eps$ under the assumption
$\left\|v\right\|_{X_T}\le\sqrt\eps$, which yields the 
desired contradiction.\\

First, we take the operator $\expp$ to the both sides of \eqref{nlkg-var}.
Using the identity $\expp \L = \pa_t \expp$
and integrating both sides with respect to time, we get the following
integral equation
\begin{align*}
 v_j(t)
 =
 \expm v_j^\circ + \int_0^t e^{i\ipax(\tau-t)}\G_j(v(\tau))\,d\tau.
\end{align*}
Taking the $H^4$ norm, we obtain
\begin{align}\label{vjH4}
 \left\|v_j(t)\right\|_{H^4}
 &\le
 \left\|v_j^\circ\right\|_{H^4}+\int_0^t 
 \left\|\G_j(v(\tau))\right\|_{H^4}\,d\tau
 \nonumber\\
 &\le
 C\eps + 
 C\int_0^t 
 \left\| F_j\left(v(\tau)+\conj v(\tau)\right)\right\|_{H^3} \,d\tau\nonumber\\
 &\le
 C\eps + C\int_0^t 
 \left\|v(\tau)\right\|_{H_\infty^1}^2
 \left\|v(\tau)\right\|_{H^3}\,d\tau
 \nonumber\\
 &\le
 C\eps + C\left\|v\right\|_{X_T}^3\int_0^t \jb{\tau}^{-1+\gamma}\,d\tau
 \le
 C\eps\jb{t}^\gamma,
\end{align}
where we used the assumption $\left\|v\right\|_{X_T}\le\sqrt{\eps}$.
Next, we use the commutation relations
\begin{align*}
\L\P = (\P-i\ipax^{-1}\pa_x)\L,\qquad
[\P,\ipax^{-1}]=-\ipax^{-3}\pa_x\pa_t
\end{align*}
to get
\begin{align}\label{LPvj1}
 \L\P v_j
 &=
 \P\L v_j - i\ipax^{-1}\pa_x \L v_j
 =
 \P\G_j(v) - i\ipax^{-1}\pa_x \G_j(v)\nonumber\\
 &=
 \frac{i}{2}\P\ipax^{-1} F_j(v+\conj v)
 +\frac{1}{2}\ipax^{-2}\pa_x F_j(v+\conj v)\nonumber\\
 &=
 \frac{i}{2}\ipax^{-1}\P F_j(v+\conj v)
 -\frac{i}{2}\ipax^{-3}\pa_x\pa_t F_j(v+\conj v)
 +\frac{1}{2}\ipax^{-2}\pa_x F_j(v+\conj v).
\end{align}
Since $\pa_t v = 4i\ipax^{-1}F (\re v)-i\ipax v$
by \eqref{nlkg-var},
the second term of \eqref{LPvj1} can be written as
\begin{align}
 &-\frac{i}{2}\ipax^{-3}\pa_x\pa_t F_j(v+\conj v)
 \nonumber\\
 &\qquad=
 -\frac{i}{2}\ipax^{-3}\pa_x\left(
 \nabla_{v+\conj v} F_j(v+\conj v) \cdot 2\re (\pa_t v)\right)
 \nonumber\\
 &\qquad=
 -i\ipax^{-3}\pa_x\bigg[
 \nabla_{v+\conj v} F_j(v+\conj v) \cdot 
 \re\left(
 4i\ipax^{-1}F (\re v)-i\ipax v\right)
 \bigg]
 \label{LPvj2}
\end{align}
where $\re Z$ stands for $(\re Z_n)_{1\le n\le N}$ for $Z\in\C^N$.
Combining \eqref{LPvj1} and \eqref{LPvj2}, we have
\begin{align}\label{LPvj3}
 \L\P v_j
 &=
 \frac{i}{2}\ipax^{-1}\P F_j(v+\conj v)
 +\frac{1}{2}\ipax^{-2}\pa_x F_j(v+\conj v)\nonumber\\
 &\qquad
 -i\ipax^{-3}\pa_x\bigg[
 \nabla_{v+\conj v} F_j(v+\conj v) \cdot 
 \re\left(
 4i\ipax^{-1}F (\re v)-i\ipax v\right)
 \bigg]\nonumber\\
 &=:
 (\L\P v_j)^{\rm I} + (\L\P v_j)^{\rm II} + (\L\P v_j)^{\rm III}.
\end{align}
Now we estimate $\L \P v_j$ in the $H^2$ norm.
By direct calculations, we get
\begin{align*}
 &\left\|(\L\P v_j)^{\rm I}\right\|_{H^2}
 \le
 C\left\|\P F_j(v+\conj v)\right\|_{H^1}
 \le
 C\left\|v\right\|_{L^\infty} \left\|v\right\|_{H_\infty^1} 
 \left\|\P v\right\|_{H^1}
 \le
 C\left\|v\right\|_{H_\infty^1}^2 \left\|\P v\right\|_{H^2},\\
 &\left\|(\L\P v_j)^{\rm II}\right\|_{H^2}
 \le
 C\left\|F_j(v+\conj v)\right\|_{H^1}
 \le
 C\left\|v\right\|_{L^\infty}^2 \left\|v\right\|_{H^1}
\end{align*}
and
\begin{align*}
 &\left\|(\L\P v_j)^{\rm III}\right\|_{H^2}\\
 &\qquad\le
 \left\|
 \nabla_{v+\conj v} F_j(v+\conj v)\cdot
 \re\left(4i\ipax^{-1} F(\re v)\right)\right\|_{L^2}
 +
 \left\|
 \nabla_{v+\conj v} F_j(v+\conj v)\cdot
 \re\left(i\ipax v\right)\right\|_{L^2}\\
 &\qquad\le
 C\left\|v\right\|_{L^\infty}^4 \left\|v\right\|_{L^2} + C\left\|v\right\|_{L^\infty}^2 \left\|v\right\|_{H^1},
\end{align*}
which yields
\begin{align}\label{LpvjH2}
 \left\|\L\P v_j\right\|_{H^2}
 &\le
 C\left\|v\right\|_{H_\infty^1}^2 \left\|\P v\right\|_{H^2}
 +C\left\|v\right\|_{H_\infty^1}^2 \left\|v\right\|_{H^4}
 +C\left\|v\right\|_{H_\infty^1}^4 \left\|v\right\|_{H^4}\nonumber\\
 &\le
 C\eps \jb{t}^{-1} \|\P v\|_{H^2} + C\eps^2 \jb{t}^{-1+\gamma},
\end{align}
where we used \eqref{vjH4} and the assumption $\left\|v\right\|_{X_T}\le \sqrt\eps$.
Again taking $\expp$ to the both sides of \eqref{LPvj3}, 
using the identity $\expp \L = \pa_t \expp$,
and integrating with respect to time,
we obtain
\begin{align}\label{Pvint}
 \P v_j(t)
 =
 \expm \P v_j(0) + \int_0^t e^{i\ipax(\tau-t)}\L\P v_j(\tau)\,d\tau.
\end{align}
And then from \eqref{LpvjH2} we have
\begin{align*}
 \left\|\P v_j(t)\right\|_{H^2}
 &\le
 \left\|\P v_j(0) \right\|_{H^2} + \int_0^t \left\|
 \L \P v_j (\tau)
 \right\|_{H^2}\,d\tau\\
 &\le
 C\eps + C\eps^2 \int_0^t \jb{\tau}^{-1+\gamma}\,d\tau
 +C\eps \int_0^t \jb{\tau}^{-1} \left\|\P v(\tau)\right\|_{H^2}\,d\tau\\
 &\le
 C\eps + C\eps\jb{t}^\gamma
 +C\eps \int_0^t \jb{\tau}^{-1} \left\|\P v(\tau)\right\|_{H^2}\,d\tau.
\end{align*}
Thus the Gronwall lemma yields
\begin{align}\label{PvH2}
 \left\|\P v(t)\right\|_{H^2}
 \le
 C\eps \jb{t}^\gamma (1+T)^{C\eps}
 \le
 C\eps \jb{t}^\gamma.
\end{align}
By the same way,
we can write the integral equation associated with the identity
$\L x = x \L - i\ipax^{-1}\pa_x$, that is,
\begin{align*}
 xv_j(t)
 =
 \expm x v_j^\circ
 +
 \int_0^t e^{i\ipax (\tau-t)} \left( x\L - i\ipax^{-1} \pa_x\right) v_j(\tau)\,d\tau.
\end{align*}
Noting that $[x, \ipax^{-1}]=-\ipax^{-3}\pa_x$, we have
\begin{align}\label{xGj}
 \left\|x\G_j (v)\right\|_{H^2}
 &\le
 C\left\|x F_j(v+\conj v)\right\|_{H^1}
 +C\left\|F_j(v+\conj v)\right\|_{L^2}\nonumber\\
 &\le
 C\left\|v\right\|_{L^\infty}^2 \left\|xv\right\|_{H^1}
 \le
 C\eps \jb{t}^{-1} \left\| xv\right\|_{H^1}.
\end{align}
So we obtain
\begin{align*}
 \left\|xv_j(t)\right\|_{H^2}
 &\le
 \left\|v_j^\circ\right\|_{H^{2,1}}
 +\int_0^t \left\| x\G_j(v(\tau))\right\|_{H^2} + 
 \left\|v_j(\tau)\right\|_{H^2}\,d\tau\\
 &\le
 C\eps + C\eps \int_0^t \jb{\tau}^\gamma\,d\tau
 +C\eps\int_0^t \jb{\tau}^{-1} \left\| x v(\tau)\right\|_{H^2}\,d\tau\\
 &\le
 C\eps \jb{t}^{\gamma+1} +
 C\eps \int_0^t \jb{\tau}^{-1} \left\| x v(\tau)\right\|_{H^2}\,d\tau.
\end{align*}
And the Gronwall lemma yields
\begin{align}\label{xvH2}
 \left\|x v(t)\right\|_{H^2}
 \le
 C\eps\jb{t}^{\gamma+1}.
\end{align}
Since $\J=i\P -ix\L - \ipax^{-1}\pa_x$, 
by \eqref{vjH4}, \eqref{PvH2}, \eqref{xGj} and \eqref{xvH2}, we get
\begin{align}\label{JvjH2}
 \left\|\J v\right\|_{H^2}
 &\le
 \left\|\P v\right\|_{H^2} + 
 \left\|x\L v\right\|_{H^2} + 
 \left\| \ipax^{-1}\pa_x v\right\|_{H^2}
 \nonumber\\
 &\le
 \left\|\P v\right\|_{H^2} + 
 \left\|x\G(v)\right\|_{H^2} 
 + \left\|v\right\|_{H^2}
 \le
 C\eps\jb{t}^\gamma.
\end{align}
From now on, we are going to evaluate $\J v$ in the $H^3$ norm.
As before, we first estimate $\left\| \P v\right\|_{H^3}$.
Noting that
\begin{align}\label{xGjH3}
 \left\|x\G_j (v)\right\|_{H^3}
 \le
 C\left\|x F_j(v+\conj v)\right\|_{H^2}
 +C\left\|F_j(v+\conj v)\right\|_{H^1}
 \le
 C\left\|v\right\|_{H_\infty^1}^2 \left\|xv\right\|_{H^2}
 \le
 C\eps^2 \jb{t}^\gamma,
\end{align}
by Lemma~\ref{lem2}, 
the relation $[\J, \ipax^2] = 2\ipax \pa_x$,
the identity $\J=i\P -ix\L - \ipax^{-1}\pa_x$
and \eqref{vjH4}, \eqref{xGjH3},
we have
\begin{align}\label{vjHi2}
 \left\|v_j\right\|_{H_\infty^2}
 &\le
 C\jb{t}^{-1/2} \left\|v_j\right\|_{H^4}^{1/2}
 \left(\|v_j\|_{H^4}^{1/2} + \left\|\J v_j\right\|_{H^3}^{1/2}\right)
 \nonumber\\
 &\le
 C\jb{t}^{-1/2} \left\|v_j\right\|_{H^4}^{1/2}
 \left(\left\|v_j\right\|_{H^4}^{1/2} + \left\|\P v_j\right\|_{H^3}^{1/2}
 +\left\|x \G_j(v)\right\|_{H^3}^{1/2}\right)
 \nonumber\\
 &\le
 C\eps \jb{t}^{\gamma-1/2}
 +C\eps^{1/2} \jb{t}^{-1/2+\gamma/2} \left\|\P v_j\right\|_{H^3}^{1/2}.
\end{align}
On the other hand, from \eqref{LPvj3} we evaluate
\begin{align*}
 \left\|\L \P v_j\right\|_{H^3}
 \le
 C\left\|v\right\|_{H_\infty^1}^2 \|\P v\|_{H^2}
 +C\left\|v\right\|_{H_\infty^1}\left\|v\right\|_{H_\infty^2}\|\P v\|_{H^1}
 +C\left\|v\right\|_{H_\infty^1}^2\left\|v\right\|_{H^2}
 +C\left\|v\right\|_{H_\infty^1}^4\left\|v\right\|_{H^1}
\end{align*}
because we have
\begin{align*}
 &\left\|(\L\P v_j)^{\rm I}\right\|_{H^3}
 \le
 C\left\|\P F_j(v+\conj v)\right\|_{H^2}
 \le
 C\left\|v\right\|_{L^\infty}^2 \left\|\P v\right\|_{H^2}
 +C\left\|v\right\|_{H_\infty^1} \left\|v\right\|_{H_\infty^2} 
 \left\|\P v\right\|_{H^1},\\
 &\left\|(\L\P v_j)^{\rm II}\right\|_{H^3}
 \le
 C\left\| F_j (v+\conj v)\right\|_{H^2}
 \le
 C\left\|v\right\|_{H_\infty^1}^2 \left\|v\right\|_{H^2}
\end{align*}
and
\begin{align*}
 &\left\|(\L\P v_j)^{\rm III}\right\|_{H^3}\\
 &\qquad\le
 \left\|
 \nabla_{v+\conj v} F_j(v+\conj v)\cdot
 \re\left(4i\ipax^{-1} F(\re v)\right)\right\|_{H^1}
 +
 \left\|
 \nabla_{v+\conj v} F_j(v+\conj v)\cdot
 \re\left(i\ipax v\right)\right\|_{H^1}\\
 &\qquad\le
 C\left\|v\right\|_{L^\infty}^4 \left\|v\right\|_{H^1} + C\left\|v\right\|_{H_\infty^1}^2 \left\|v\right\|_{H^2}.
\end{align*}
Thus by \eqref{vjHi2}, \eqref{vjH4}, \eqref{PvH2} 
and $\left\|v\right\|_{X_T}\le\sqrt\eps$, 
we evaluate that
\begin{align*}
 \left\|\L\P v_j\right\|_{H^3}
 &\le
 C\eps \jb{t}^{-1} \left\|\P v\right\|_{H^2}
 +C\eps^2 \jb{t}^{-1+\gamma}
 +C\eps^{3/2}\jb{t}^{-1/2+\gamma}\left\|v\right\|_{H_\infty^2}
 \nonumber\\
 &\le
 C\eps \jb{t}^{-1} \left\|\P v\right\|_{H^3}
 +C\eps^2 \jb{t}^{-1+2\gamma}
 +C\eps^2 \jb{t}^{-1+3\gamma/2}\|\P v\|_{H^3}^{1/2}
 \nonumber\\
 &\le
 C\eps \jb{t}^{-1} \|\P v\|_{H^3}
 +C\eps^2 \jb{t}^{-1+2\gamma}
 +C\eps^2\left(\jb{t}^{-1}
 \left\|\P v\right\|_{H^3} + \jb{t}^{-1+3\gamma}\right)
 \nonumber\\
 &\le
 C\eps \jb{t}^{-1} \left\|\P v\right\|_{H^3}
 +C\eps^2 \jb{t}^{-1+3\gamma}.
\end{align*}
Taking the $H^3$ norm to \eqref{Pvint}, we obtain
\begin{align*}
 \left\|\P v_j(t)\right\|_{H^3}
 &\le
 \left\|\P v_j(0)\right\|_{H^3}+\int_0^t 
 \left\|\L\P v_j(\tau)\right\|_{H^3}\,d\tau\\
 &\le
 C\eps + C\eps^2 \int_0^t \jb{\tau}^{-1+3\gamma}\,d\tau
 +C\eps \int_0^t \jb{\tau}^{-1}\left\|\P v(\tau)\right\|_{H^3}\,d\tau\\
 &\le 
 C\eps\jb{t}^{3\gamma}
 +C\eps \int_0^t \jb{\tau}^{-1}\left\|\P v(\tau)\right\|_{H^3}\,d\tau.
\end{align*}
Therefore the Gronwall lemma yields
\begin{align}\label{PvH3}
 \left\|\P v(t)\right\|_{H^3}
 \le
 C\eps\jb{t}^{3\gamma}.
\end{align}
By the identity $\J=i\P - ix\L - \ipax^{-1}\pa_x$, 
\eqref{vjH4}, \eqref{xGjH3} and \eqref{PvH3},
we see that
\begin{align}\label{JvH3}
 \left\|\J v(t)\right\|_{H^3}
 \le
 C\eps \jb{t}^{3\gamma}.
\end{align}
Now we are in a position to estimate $\left\| v\right\|_{H_\infty^1}$.
To do this, we will derive an $L^\infty$ estimate for
the new variable $\psi(t) = \jb{\xi}^\kappa\F \expp v(t)$
where $\frac{3}{2}\le\kappa\le 4$,
and then prove $\left\|v(t)\right\|_{H_\infty^1}\le C\eps\jb{t}^{-1/2}$
by the decomposition of the free Klein-Gordon evolution group.
We note that in the case of $t\le 1$,
the standard Sobolev embedding and \eqref{vjH4} lead to
\begin{align*}
 \jb{t}^{1/2+\gamma}\left\|v(t)\right\|_{H_\infty^1}
 \le
 2^{1/4+\gamma/2}\left\|\ipax v(t)\right\|_{L^\infty}
 \le
 C\left\|\ipax v(t)\right\|_{H^1}
 \le
 C\eps\jb{t}^\gamma,
\end{align*}
so that
\begin{align}\label{smallt}
 \left\|v(t)\right\|_{H_\infty^1}\le C\eps\jb{t}^{-1/2}
\end{align}
holds for $t\le 1$. From now on, we focus on the case of $t\ge 1$.
We recall Lemma~\ref{lem4} for $\psi(t)$:
\begin{align}\label{psiode1}
 \pa_t \psi (t) 
 =
 \frac{1}{2} it^{-1}\jb{\xi}^{2-2\kappa}\widetilde{F}(\psi(t))+S(t)+R(t),
 \qquad t\ge1,
\end{align}
where the non-resonant term $S(t)$ takes the form of
\begin{align*}
 S_j(t)
 &=
  it^{-1}
 \sum_{k,l,m=1}^N\sum_{\nu=1,4,6,7,8}
 C_{j,k,l,m}^\nu
 a_{\nu,\kappa}(\xi)
 e^{-itb_\nu(\xi)}
  \D_{\omega_\nu}
 \prod_{\lambda=k,l,m}\left(
 \conj{\psi_\lambda}^{1-\alpha_{\lambda_\nu}}
 \psi_\lambda^{\alpha_{\lambda_\nu}}\right)
\end{align*}
with some constants $C_{j,k,l,m}^\nu\in\C$,
$\omega_\nu$, $\alpha_{\lambda_\nu}$ given in Lemma~\ref{lem3}
and
\begin{align*}
 a_{\nu,\kappa}(\xi)
 &=
 \jb{\xi}^{\kappa-1}
 \langle\xi\omega_\nu^{-1}\rangle^{3-3\kappa},\\
 b_\nu(\xi)
 &=
 \omega_\nu \langle\xi \omega_\nu^{-1}\rangle-\jb{\xi}.
\end{align*}
Here the remainder $R(t)$ satisfies the estimate
\begin{align}\label{Remainder}
 \left\|R_j(t)\right\|_{L^\infty}
 &\le
 Ct^{-5/4}\left\|\psi(t)\right\|_{H^{1,4-\kappa}}^3
 \nonumber\\
 &\le
 Ct^{-5/4}\left\|\jb{\xi}\F \expp v(t)\right\|_{H^{1,3}}^3
 \nonumber\\
 &\le
 C\eps t^{-5/4 + 9\gamma},
\end{align}
where we used along the definition of $\J$,
\begin{align}\label{xifexpv}
 \left\|\jb{\xi}\F\expp v(t)\right\|_{H^{1,3}}
 &\le
 C\left\|\jb{\xi}^3 \jb{\xi}\pa_\xi \F\expp v(t)\right\|_{L^2}
 +C\left\|\jb{\xi}^4\F\expp v(t)\right\|_{L^2}\nonumber\\
 &\le
 C\left\|\jb{\xi}^3 \jb{\xi}\expmxi i\pa_\xi \exppxi\F v(t)\right\|_{L^2}
 +C\left\|v(t)\right\|_{H^4}\nonumber\\
 &\le
 C\left\|\J v(t)\right\|_{H^3}
 +C\left\| v(t)\right\|_{H^4}\nonumber\\
 &\le
 C\eps\jb{t}^{3\gamma}
\end{align}
for the last inequality.
Here we note that
\begin{align}\label{matA}
 |Y\cdot AZ|^2 \le (Y\cdot AY)(Z\cdot AZ),
 \qquad
 c_* |Y|^2 \le Y\cdot AY \le c^* |Y|^2
\end{align}
for any $Y,Z\in\C^N$, where the matrix $A$ is in Theorem~\ref{thm2}
and $c^*$ (resp. $c_*$) is the largest (resp. smallest) eigenvalue of $A$.
Then it follows from
\eqref{psiode1}, \eqref{struc-condi}, \eqref{matA}
and \eqref{Remainder} that
\begin{align*}
 &\pa_t \left(\psi(t) \cdot A\psi(t)\right)
 =
 2\re\left(\pa_t\psi(t) \cdot A\psi(t)\right)\\
 &\qquad=
 -t^{-1}\jb{\xi}^{2-2\kappa}
 \im\left(\widetilde{F}(\psi(t))\cdot A\psi(t)\right)
 +2\re\left(S(t)\cdot A\psi(t)\right)
 +2\re\left(R(t)\cdot A\psi(t)\right)\\
 &\qquad\le
 2\re\left(S(t)\cdot A\psi(t)\right)
 +2t^{-5/4}\left|t^{5/4}R(t)\cdot A\psi(t)\right|\\
 &\qquad\le
 2\re\left(S(t)\cdot A\psi(t)\right)
 +t^{-5/4}\left(\psi(t)\cdot A\psi(t)
 +t^{5/2} R(t)\cdot AR(t)\right)\\
 &\qquad\le
 2\re\left(S(t)\cdot A\psi(t)\right)
 +t^{-5/4}\left(\psi(t)\cdot A\psi(t)\right)
 +C\eps^2 t^{-5/4+18\gamma}.
\end{align*}
Noting that (as in the proof of \eqref{xifexpv})
\begin{align*}
 \left\|\psi(1)\right\|_{L^\infty}
 \le
 C\left\|\jb{\xi}^\kappa \F \expp v(1)\right\|_{H^1}
 \le
 C\left\|\jb{\xi} \F \expp v(1)\right\|_{H^{1,\kappa-1}}
 \le
 C\eps
\end{align*}
holds, we obtain
\begin{align*}
 \psi(t)\cdot A\psi(t)
 &\le
 C\eps^2
 +2\left|\int_1^t \re\left(S(\tau)\cdot A\psi(\tau)\right)\,d\tau \right|
 +\int_1^t \tau^{-5/4} \left(\psi(\tau)\cdot A\psi(\tau)\right)\,d\tau
 \nonumber\\
 &\le
 C\eps^2
 +\int_1^t \tau^{-5/4} \left(\psi(\tau)\cdot A\psi(\tau)\right)\,d\tau
\end{align*}
for $t\ge1$, provided that
\begin{align}\label{intres}
 \sup_{t\in[1,T]}
 \left|\int_1^t \re\left(S(\tau)\cdot A\psi(\tau)\right) \,d\tau\right|
 \le
 C\eps^2.
\end{align}
Once we get \eqref{intres}, we can apply the Gronwall lemma
and \eqref{matA} to obtain
\begin{align}\label{psiLinfty}
 \sup_{t\in[1,T]}\left\|\psi(t)\right\|_{L^\infty}
 \le
 C\eps.
\end{align}
In order to establish \eqref{intres}, we observe that
\begin{align*}
 &\int_1^t 
 \frac{a_{\nu,\kappa}(\xi)e^{-i\tau b_\nu(\xi)}}{\tau}
 \psi_{j_1}^{(\sigma_1)}(\tau,\xi)
 \prod_{n=2}^4 \left(\psi_{j_n}^{(\sigma_n)}
 (\tau,\xi\omega_\nu^{-1})\right)\,d\tau\\
 &\qquad=
 \int_1^t \left(\pa_\tau \Psi_1\right)(\tau,\xi) + \Psi_2 (\tau,\xi)\, d\tau
 =
 \Psi_1(t,\xi) - \Psi_1(1,\xi) + \int_1^t \Psi_2(\tau,\xi)\,d\tau,
\end{align*}
for $\nu\in\{1,4,6,7,8\}$ (i.e. $\omega_\nu\in\{-1,\pm3\}$),
$j_1, \ldots, j_4\in\{1,\ldots,N\}$
and
$\sigma_1, \ldots, \sigma_4 \in \{+,-\}$,
where
$\psi_{j_n}^{(+)} = \psi_{j_n}$,
$\psi_{j_n}^{(-)} = \conj{\psi_{j_n}}$
and
\begin{align*}
 &\Psi_1(\tau,\xi)
 =
 \frac{a_{\nu,\kappa}(\xi) e^{-i\tau b_\nu (\xi)}}{
 -i\tau b_\nu(\xi)}
 \psi_{j_1}^{(\sigma_1)}(\tau,\xi)
 \prod_{n=2}^4 \left(\psi_{j_n}^{(\sigma_n)} 
 (\tau, \xi\omega_\nu^{-1})\right),\\
 &\Psi_2(\tau,\xi)
 =
 \frac{a_{\nu,\kappa}(\xi) e^{-i\tau b_\nu (\xi)}}{
 -i\tau^2 b_\nu(\xi)}
 \psi_{j_1}^{(\sigma_1)}(\tau,\xi)
 \prod_{n=2}^4 \left(\psi_{j_n}^{(\sigma_n)} 
 (\tau, \xi\omega_\nu^{-1})\right)\\
 &\qquad\qquad\qquad+
 \frac{a_{\nu,\kappa}(\xi) e^{-i\tau b_\nu (\xi)}}{
 i\tau b_\nu(\xi)}
 \Bigg[
 \big(\pa_\tau \psi_{j_1}^{(\sigma_1)}\big)(\tau,\xi)
 \prod_{n=2}^4 \left(\psi_{j_n}^{(\sigma_n)} 
 (\tau, \xi\omega_\nu^{-1})\right)\\
 &\qquad\qquad\qquad+
 \psi_{j_1}^{(\sigma_1)}(\tau,\xi)\sum_{n=2}^4
 \bigg(
 \big(\pa_\tau \psi_{j_n}^{(\sigma_n)}\big)(\tau,\xi\omega_\nu^{-1})
 \prod_{\varsigma\in\{2,3,4\}\setminus\{n\}}
 \psi_{j_\varsigma}^{(\sigma_\varsigma)}(\tau,\xi\omega_\nu^{-1})\bigg)
 \Bigg].
\end{align*}
Using \eqref{psiode1}, \eqref{Remainder}
and
\begin{align*}
 \left\|\psi(\tau)\right\|_{L^\infty}
 \le
 C\left\|\jb{\xi}^\kappa \F \expp v(\tau)\right\|_{H^1}
 \le
 C\left\|\jb{\xi} \F \expp v(\tau)\right\|_{H^{1,\kappa-1}}
 \le
 C\eps \jb{\tau}^{3\gamma},
\end{align*}
we have
\begin{align*}
 \left\|\pa_\tau \psi (\tau)\right\|_{L^\infty}
 &\le
 C\tau^{-1}
 \left(\jb{\xi}^{2-2\kappa}
 +\jb{\xi}^{\kappa-1}
 \langle\xi\omega_\nu^{-1}\rangle^{3-3\kappa}
 \right)\left\|\psi(\tau)\right\|_{L^\infty}^3
 +C\eps\tau^{-5/4+9\gamma}\\
 &\le
 C\eps^3\tau^{-1}\jb{\tau}^{9\gamma}
 +C\eps\tau^{-5/4+9\gamma}
 \le
 C\eps\tau^{-1+9\gamma}
\end{align*}
since $\frac{3}{2}\le \kappa\le 4$.
Observing that when $\kappa\ge\frac{3}{2}$,
\begin{align*}
 \left|\frac{a_{\nu,\kappa}(\xi)}{b_\nu(\xi)}\right|
 =
 \left|\frac{\jb{\xi}^{\kappa-1}\jb{\xi\omega_\nu^{-1}}^{3-3\kappa}}
 {\omega_\nu\jb{\xi\omega_\nu^{-1}}-\jb{\xi}}\right|
 \le
 \frac{\jb{\xi}^{\kappa-1}\jb{\xi\omega_\nu^{-1}}^{3-3\kappa}}
 {\jb{\xi}^{-1}}
 \le C
\end{align*}
for all $\xi\in\R$ and $\omega_\nu\in\{-1,\pm 3\}$, we get
\begin{align*}
 &|\Psi_1(\tau,\xi)|
 \le
 C\tau^{-1}
 \left\|\psi(\tau)\right\|_{L^\infty}^4
 \le
 C\eps^4 \tau^{-1+12\gamma}
\end{align*}
and
\begin{align*}
 |\Psi_2(\tau,\xi)|
 &\le
 C\tau^{-2}\left\|\psi(\tau)\right\|_{L^\infty}^4
 +C\tau^{-1}\left\|\pa_\tau \psi(\tau)\right\|_{L^\infty}
 \left\|\psi(\tau)\right\|_{L^\infty}^3\\
 &\le
 C\eps^4 \tau^{-2+12\gamma}
 +C\eps^4 \tau^{-2+18\gamma}
 \le
 C\eps^4 \tau^{-2+18\gamma}.
\end{align*}
From them we deduce that
\begin{align*}
 &\sup_{t\in[1,T]}
 \left|\int_1^t \re\left(S(\tau)\cdot A\psi(\tau)\right) \,d\tau\right|\\
 &\qquad\le
 \sum_{\substack{\nu\in\{1,4,6,7,8\}\\
 j_1,\ldots,j_4\in\{1,\ldots,N\}\\
 \sigma_1,\ldots,\sigma_4\in\{+,-\}}}
 C\sup_{t\in[1,T]}
 \left|
 \int_1^t 
 \frac{a_{\nu,\kappa}(\xi) e^{-i\tau b_\nu(\xi)}}{\tau}
 \psi_{j_1}^{(\sigma_1)}(\tau,\xi)
 \prod_{n=2}^4 \psi_{j_n}^{(\sigma_n)}
 (\tau,\xi\omega_\nu^{-1})\,d\tau
 \right|\\
 &\qquad\le
 C\eps^4 \left(1+\int_1^\infty \tau^{-2+18\gamma}\,d\tau\right).
\end{align*}
This proves \eqref{intres} so that \eqref{psiLinfty} follows immediately.
Now we are ready to prove $\left\|v(t)\right\|_{H_\infty^1}\le C\eps\jb{t}^{-1/2}$
for $t\ge 1$.
We let $\psi^1 (t) = \jb{\xi}\F\expp v(t)$ as before.
By the decomposition of the free Klein-Gordon evolution group
\eqref{decompKG}
and the relation $v(t)=\ipax^{-1}\expm\IF\psi^1(t)$,
we have
\begin{align*}
 &\ipax v(t)
 =
 \expm \IF \psi^1(t)\\
 &\quad=
 \left(\D_t \M(t) \B +\D_t \M(t) \B(\V(t)-1) + \D_t \W(t)\right)\psi^1(t)\\
 &\quad=
 \D_t \M(t)\B\jb{\xi}^{-3/2}\jb{\xi}^{3/2}\psi^1(t) +
 \D_t \M(t)\B\jb{\xi}^{-3/2}\jb{\xi}^{3/2}(\V(t)-1)\psi^1(t)
 +\D_t\W(t)\psi^1(t).
\end{align*}
Since \eqref{psiLinfty} implies 
$\|\jb{\xi}^{3/2}\psi^1 (t)\|_{L^\infty}\le C\eps$ when $\kappa=\frac{5}{2}$,
from the second and the third estimates in Lemma~\ref{lem1},
we finally obtain
\begin{align*}
 \left\|v(t)\right\|_{H_\infty^1}
 &\le
 Ct^{-1/2}\bigg(\left\|\jb{\xi}^{3/2}\psi^1(t)\right\|_{L^\infty}
 +\left\|\jb{\xi}^{3/2} (\V(t)-1) \psi^1(t)\right\|_{L^\infty}
 +\left\|\W(t)\psi^1(t)\right\|_{L^\infty}\bigg)\\
 &\le
 C\eps t^{-1/2}
 +Ct^{-1/2-1/4} \left\|\psi^1(t)\right\|_{H^{1,3}}
 +Ct^{-1/2-1/2} \left\|\psi^1(t)\right\|_{H^{1,3}}\\
 &\le
 C\eps t^{-1/2} + Ct^{-3/4} \left\|\jb{\xi}\F\expp v(t)\right\|_{H^{1,3}}\\
 &\le
 C\eps t^{-1/2} + C\eps t^{-3/4+3\gamma}
\end{align*}
for $t\ge 1$, 
where we used \eqref{xifexpv} for the last inequality.
At last together with \eqref{smallt} we reach
\begin{align}\label{vHinfty1}
 \left\|v(t)\right\|_{H_\infty^1}
 \le
 C\eps\jb{t}^{-1/2}
\end{align}
for all $t\in[0,T]$.
Therefore it follows from
\eqref{vjH4}, \eqref{JvjH2}, \eqref{JvH3} and \eqref{vHinfty1} that
\begin{align*}
 \left\| v\right\|_{X_T} \le C\eps <\sqrt\eps
\end{align*}
for sufficiently small $\eps$,
which implies the desired contradiction.
Thus there exists a unique global solution
$v\in C^0([0,\infty);H^{4,1})$ of the initial value problem \eqref{nlkg-var}
with the time-decay estimate \eqref{vdecay}
under the condition \eqref{struc-condi}.
The proof of Theorem~\ref{thm2} is completed.\qed

\end{document}